\documentclass[11pt]{article}
\usepackage[affil-it]{authblk}
\usepackage{a4wide}
\usepackage{amsthm}
\usepackage{amsmath,amssymb,amscd,eucal}
\usepackage{mathtools}
 \usepackage{latexsym}
\usepackage{graphicx}
\usepackage[usenames,dvipsnames,svgnames,table]{xcolor}
\usepackage{color} 
\usepackage{cancel}
\usepackage{enumitem}
\usepackage[normalem]{ulem}
\usepackage{varwidth}

\newtheorem{thm}[equation]{Theorem}
\newtheorem{cor}[equation]{Corollary}
\newtheorem{lemma}[equation]{Lemma}
\newtheorem{prop}[equation]{Proposition}

\newtheorem*{conj*}{Conjecture}

\theoremstyle{definition}

\newtheorem{remark}[equation]{Remark}

\numberwithin{equation}{section}

\newcommand{\FF}{\mathbb{F}}  
  
\newcommand{\ZZ}{\mathbb{Z}}

\newcommand{\BB}{\mathsf{B}}

\newcommand\chara{\mathsf{char}}
\newcommand{\lb}[1]{\left[ #1\right]}
\newcommand{\pb}[1]{\left\{ #1\right\}}
\newcommand{\seq}[1]{\left( #1\right)}

\newcommand{\ann}{\hbox{\rm Ann}}

\newcommand\Z[1]{\mathsf{Z}^2_{\mathrm{#1}}}
\newcommand\B[1]{\mathsf{B}^2_{\mathrm{#1}}}
\newcommand\HH[1]{\mathsf{H}^2_{\mathrm{#1}}}

\newcommand\Aut{\operatorname{\mathsf{Aut}}}
\newcommand\orb{\operatorname{\mathsf{Orb}}}

\newcounter{marg}[section]

\allowdisplaybreaks
 
\title{Non-associative central extensions of null-filiform associative algebras\thanks{The first part of this work is supported by the Russian Science Foundation under grant 18-71-10007. 
The second part of this work was supported by MTM 2016-79661-P; FPU scholarship (Spain); and CMUP (UID/MAT/00144/2019), which is funded by FCT (Portugal) with national (MCTES) and European structural funds through the programs FEDER, under the partnership agreement PT2020.}}

\author[a, b]{Ivan Kaygorodov}
\author[c]{Samuel A.\ Lopes}
\author[d]{Pilar P\'{a}ez-Guill\'{a}n}

\affil[a]{CMCC, Universidade Federal do ABC, Santo Andr\'e, Brazil.}
\affil[b]{Siberian Federal University, Krasnoyarsk, Russia.}
\affil[c]{CMUP, Faculdade de Ci\^encias da
Universidade do Porto, Rua do Campo Alegre 687,
4169-007 Porto, Portugal.}
\affil[d]{IMAT, Universidade de Santiago de Compostela, Santiago de Compostela, Spain.}

\affil[ ]{Email addresses: \textit{kaygorodov.ivan@gmail.com, slopes@fc.up.pt, pilar.paez@usc.es}}
\date{}

\begin{document}  

\maketitle

\begin{abstract}
We give the algebraic classification of alternative, left alternative, Jordan, bicommutative, left commutative, assosymmetric, Novikov and left symmetric  central extensions of null-filiform associative algebras. 
\newline\newline
\textbf{MSC Numbers (2010)}: 
16D70, 16S70.
 \hfill \newline
\end{abstract}

\section*{Introduction}\label{S:intro}

The subject of algebra extensions has long been a focal point of interest in mathematics and physics. In short, an algebra ${\rm A}$ is an extension of another algebra $ {\rm B}$ by ${\rm K}$ if there exists a short exact sequence $0\to {\rm K}\to {\rm A}\to {\rm B}\to 0$. The easiest example is the direct sum ${\rm B}\oplus {\rm K}$, along with the inclusion and projection maps.
Imposing new conditions on the extensions, we find important special types, such as the split extensions, the HNN-extensions, which have been used to prove numerous theorems of embeddability in different varieties of non-associative algebras \cite{L1,LS,wasserman}, and many others. In this paper, we will be concerned with central extensions, i.e.\ extensions in which the annihilator of ${\rm A}$ contains ${\rm K}$.
Several important algebras can be constructed as central extensions; for example, the Virasoro
algebra is the universal central extension of the Witt algebra, and the Heisenberg algebra is
a central extension of a commutative Lie algebra.

The algebraic study of central extensions of different varieties of non-associative algebras plays  an important role in the classification problem in such varieties (for example, see \cite{zusmanovich}). 
Since Skjelbred and Sund in 1978 \cite{ss78} devised a method for classifying nilpotent Lie algebras, making crucial use of central extensions, it has been adapted to many other varieties of algebras, including associative \cite{degr1},  Malcev \cite{hac16,hac18}, Jordan \cite{ha16,ha17}, Novikov \cite{kkk18}, anticommutative \cite{cfk182},
binary Lie \cite{ack} or bicommutative \cite{kpv19}, among others
(see, \cite{degr3,usefi1,degr2,fkkv19,gkk}).

The study of central extensions of null-filiform algebras was initiated  in \cite{omirov},
where all Leibniz central extensions of null-filiform Leibniz algebras were described. 
The associative central extensions of the associative null-filiform algebra of dimension $n$, which we will denote by $\mu_0^n$, were studied in \cite{kkl18} within the framework of associative algebras, and it was proven that the only associative central extension of $\mu_0^n$ is $\mu_0^{n+1}$. However, the null-filiform algebras can be considered as elements of more general varieties of algebras, such as alternative, left alternative, Jordan, bicommutative, left commutative, assosymmetric, Novikov or left symmetric, among others (note that the right alternative, right commutative and right symmetric cases are analogous to their respective left counterparts). In particular, it was proven in~\cite{burde,kpv19} that, up to isomorphism, the null-filiform algebra  $\mu_0^3$ admits the trivial extension $\mu_0^4$ and also another nontrivial bicommutative extension. Then, it is reasonable to wonder whether there will be nontrivial extensions in the aforementioned varieties of algebras.

Our main result is the classification of the isomorphism classes of central extensions
of the associative null-filiform algebra over several varieties of non-associative algebras, as the
ones mentioned above. This is in part summarized in Table~\ref{table1}, at the end of Section~\ref{S:lbc}.

The structure of the paper is as follows. 
Section 1 is devoted to formalizing the method of Skjelbred and Sund over any variety of non-associative algebras defined by a set of polynomial identities. 
Section 2 presents a quick review of null-filiform algebras. 
Respectively, Sections 3, 4 and 5 deal with left alternative and alternative, Jordan, and left commutative and bicommutative central extensions. 
Finally, Section 6 deals with the assosymmetric, Novikov and left symmetric cases, which happen to come out as a trivial corollary of the left commutative and bicommutative cases.

\medskip 

\paragraph{\bf Notation and underlying assumptions} Throughout the paper, $\FF$ will denote a field. Unless otherwise specified, this field will be arbitrary and all vector spaces, tensor products, (multi)linear maps and automorphism groups will be taken over $\FF$. Given a nonnegative integer $n$, $[n]$ will denote the set $\pb{i\in \ZZ\mid 1\leq i\leq n}$.
Also, we fix the following notation:
\[
\begin{array}{rl}
 {\rm AS}: & \mbox{variety of associative algebras;}\\
 {\rm AL}: & \mbox{variety of left alternative algebras;}\\
 {\rm J}: & \mbox{variety of Jordan algebras;}\\
 {\rm LC}: & \mbox{variety of left commutative algebras;}\\
  {\rm RC}: & \mbox{variety of right commutative algebras;}\\
  {\rm BC}: & \mbox{variety of bicommutative algebras.}
\end{array}
\]





\section{Non-associative central extensions}\label{S:general}

A variety ${\rm M}$ of non-associative algebras over $\FF$ is defined by a set of identities $\{E_i\}_{i\in I}$ of the form \[E_i\colon\ \sum_{j=1}^{t_{i}}p_{i,j}(z_1, \ldots, z_\ell)=0,
\]
where $Z=\pb{z_1, \ldots, z_\ell}$ is a finite alphabet, $p_{i, j}=p_{i, j}(z_1, \ldots, z_\ell)$ is a non-associative word in $Z$ of length $n_{i, j}\geq 2$, with a coefficient of either $1$ or $-1$ and $Z_{i,j}\subseteq Z$ is the set of letters which occur in $p_{i, j}$. 

As $n_{i,j}\geq 2$, each $p_{i,j}$ can be expressed uniquely as a product
\[p_{i,j}=p_{i,j}^1p_{i,j}^2,\]
where $Z_{i,j}=Z_{i,j}^1\cup Z_{i,j}^2$ and each $p_{i,j}^k$, with $k\in\{1,2\}$, either has length one or is another concatenation of products. Recursively, we will obtain $n_{i,j}-1$ factors $p_{i,j}^{\alpha}$ with $|Z_{i,j}^{\alpha}|=1$ and $\alpha$ of the form $\alpha=\alpha_1,\alpha_2,\alpha_3\ldots,\alpha_\ell$, with $\alpha_k\in\{1,2\}$ for all $k$.

Let ${\rm A}$ be an algebra in the variety ${\rm M}$, and let ${\rm V}$ be a vector space over $\FF$. Following the Skjelbred-Sund method, we introduce the \textit{cocycles} of ${\rm A}$ with respect to ${\rm V}$ as the bilinear maps $\theta\colon {\rm A}\times {\rm A}\longrightarrow {\rm V}$ satisfying the set of identities $\{\tilde{E}_i\}_{i\in I}$, where 
\[
\tilde{E}_i\colon\ \sum_{j=1}^{t_{i}}\theta(p_{i,j}^1,p_{i,j}^2)=0.\] These elements form a vector space over $\mathbb{F}$, which we will denote by $\Z{{\rm M}}({\rm A},{\rm V})$.

We also define the \textit{coboundaries} of ${\rm A}$ with respect to ${\rm V}$ as follows. Let $f$ be a linear map from ${\rm A}$ to ${\rm V}$, and set $\delta f\colon {\rm A}\times {\rm A}\longrightarrow {\rm V}$ with $(\delta f)(x,y)=f(xy)$. It is clear from
\[
\sum_{j=1}^{t_i}\delta f(p_{i,j}^1,p_{i,j}^2)=\sum_{j=1}^{t_i}f(p_{i,j})=f(0)=0,\]
for all $i\in I$, that $\delta f\in \Z{{\rm M}}({\rm A},{\rm V})$. 
We define $\B{}  ({\rm A},{\rm V})=\{\delta f :   f\in {\rm Hom}( {\rm A},{\rm V} )\}$;  it is a linear subspace of $\Z{M}({\rm A},{\rm V})$.  The quotient space $\Z{M}({\rm A},{\rm V}) / \B{}({\rm A},{\rm V})$ is called the {\it second cohomology} space and is denoted by $\HH{M}({\rm A},{\rm V})$.

Consider the group $\Aut({\rm A})$ of automorphisms of the algebra ${\rm A}$. If $\theta$ is a cocycle and $\phi\in\Aut({\rm A})$, we define $\phi\cdot\theta : {\rm A} \times {\rm A} \longrightarrow {\rm V}$ by $\phi\cdot\theta(x,y)=\theta (\phi(x),\phi(y))$. The automorphism $\phi$ preserves the product, so $\phi(p_{i,j}(x_1,\dots,x_{\ell}))=p_{i,j}(\phi(x_1),\dots,\phi(x_{\ell}))$; from this, we have that $\phi\cdot\theta\in\Z{M}({\rm A},{\rm V})$. This induces an action of $\Aut( {\rm A})$ on $\Z{M}({\rm A},{\rm V})$. Moreover, $\theta=\delta f\iff \phi\cdot\theta=\delta(f\circ\phi)$, so the action is inherited by the quotient $\HH{M}({\rm A},{\rm V})$. Although this is a right action, we will write it on the left to follow the usual convention. The orbit of an element $\theta\in\Z{M}({\rm A},{\rm V})$ will be denoted by $\orb{(\theta)}$, and the orbit of $[\theta]\in\HH{M}({\rm A},{\rm V})$, by $\orb{([\theta])}$.

For every bilinear map $\theta\colon {\rm A} \times {\rm A} \longrightarrow {\rm V}$, we can define the algebra ${\rm A}_{\theta}={\rm A}\oplus {\rm V}$ with the product $[x+v,y+w]_{\theta}=xy+\theta(x,y)$.

\begin{lemma}
The algebra ${\rm A}_{\theta}$ belongs to the variety ${\rm M}$ if and only if $\theta\in\Z{M}({\rm A},{\rm V}).$
\end{lemma}

\begin{proof}
By definition, ${\rm A}_{\theta}$ belongs to ${\rm M}$ if it satisfies the identities $\{E_i\}_{i\in I}$,
i.e., if it satisfies
\[\sum_{j=1}^{t_{i}}p_{i,j}(x_1+v_1,\dots,x_{\ell}+v_{\ell})=0,\]
for all $i\in I$ and all $x_k\in{\rm A}$, $v_k\in{\rm V}$, with $k\in[\ell]$. For the sake of brevity, we will write
\[\sum_{j=1}^{t_{i}}p_{i,j}(X+V)=0;\]
also, when we make reference to these identities in the algebra ${\rm A}$, we will write
\[\sum_{j=1}^{t_{i}}p_{i,j}(X)=0.\]
It holds that 
\[p_{i,j}(X+V)=[p^{1}_{i,j}(X+V),p^{2}_{i,j}(X+V)]_{\theta}.\]
An easy induction in the cardinal $|Z_{i,j}|$ shows that \[p_{i,j}(X+V)=p_{i,j}(X)+\theta(p^1_{i,j}(X),p^2_{i,j}(X)).\]
As ${\rm A}$ belongs to the variety ${\rm M}$, it follows that ${\rm A}_{\theta}$ satisfies the identities $\{E_i\}_{i\in I}$ if and only if $\theta$ satisfies $\{\tilde{E}_i\}_{i\in I}$.
\end{proof}

In all that follows, we will consider ${\rm A}_{\theta}$ just for $\theta\in\Z{M}({\rm A}, {\rm V})$. Then, it is easy to check that  the algebra ${\rm A}_{\theta}$ of the variety ${\rm M}$ is a central extension of ${\rm A}$ by $V$. We define the dimension of the extension as the dimension of ${\rm V}.$

The particular identities of the variety ${\rm M}$ are not involved any more in the development of the method. Thus, we will just give a brief overview of the main definitions and results, and refer the reader to other texts in which this method is carefully detailed, such as \cite{ha16, hac16}.

One fundamental definition is the {\it annihilator} of $\theta$, 
$\ann(\theta)=\{x\in{\rm A}\mid \theta(x,{\rm A})+\theta ({\rm A}, x)=0\}$. Recalling that the annihilator of an algebra ${\rm A}$ is $\ann ({\rm A})=\{x\in{\rm A}\mid x{\rm A }+ {\rm A}x=0 \}$, it is clear that $\ann({\rm A}_{\theta})=(\ann({\rm A})\cap \ann(\theta))\oplus {\rm V}$. 

Other important notions are the following. 
Let ${\rm A}$ be an algebra and ${\rm I}$ be a nonzero subspace of $\operatorname{Ann}({\rm A})$. If ${\rm A}={\rm A}_0 \oplus {\rm I}$ for some ideal ${\rm A}_0$,
then ${\rm I}$ is called an {\it annihilator component} of ${\rm A}$.
The central extensions of ${\rm A}$ with no annihilator component are called {\it non-split central extensions}.

We show now that any algebra ${\rm A}$ of the variety ${\rm M}$ with nonzero annihilator is isomorphic to a central extension of some other suitable algebra ${\rm A'}$ in ${\rm M}$. The demonstration can be found in~\cite[Lemma 5]{hac16}.

\begin{lemma}\label{lm:first}
Let ${\rm A}$ be an $n$-dimensional algebra in the variety ${\rm M}$ such that $\dim(\ann({\rm A}))=s\neq0$. Then there exist, up to isomorphism, a unique $(n-s)$-dimensional algebra ${\rm A}'$ in $\rm M$ and a cocycle $\theta \in \Z{M}({\rm A}, {\rm V})$ for some vector space ${\rm V}$ of dimension $s$, such that $\ann({\rm A})\cap\ann(\theta)=0$,  ${\rm A} \cong {{\rm A}'}_{\theta}$ and
 ${\rm A}/\ann({\rm A})\cong {\rm A}'$.
\end{lemma}

Then, in order to decide when two algebras with nonzero annihilator of the variety $\rm M$ are isomorphic, it suffices to find criteria in terms of the cocycles.

Let us fix a basis $\{e_{1},\ldots ,e_{s}\}$ of ${\rm V}$. For every algebra $A$ and every cocycle $\theta\in \Z{M}({\rm A},{\rm V})$, there exist $s$ unique cocycles $\theta_i\in \Z{M}({\rm A},{\FF})$, $i\in[s]$, such that $\theta(x,y)=\sum_{i=1}^{s}\theta_i(x,y)e_i$, for all $x,y\in{\rm A}$. Straightforward facts about these cocycles are that $\theta\in \BB^2({\rm A},{\rm V})$ if and only if $\theta_i\in \BB^2({\rm A},{\FF})$ for $i\in [s]$, and that $\ann (\theta)=\ann(\theta_i)\cap\dots\cap \ann(\theta_s)$.
Also, it is not difficult to prove (see \cite[Lemma 13]{hac16}) that if $\ann(\rm A)\cap \ann (\theta)=0$, then ${\rm A}_{\theta}$ has an annihilator component if and only if $[\theta_1],\dots,[\theta_s]$ are linearly dependent in $\HH{M}({\rm A},{\FF})$. In other words, if $\ann(\rm A)\cap \ann (\theta)=0$, then ${\rm A}_{\theta}$ is a non-split central extension of ${\rm A}$ if and only if $\langle [\theta_1],\dots,[\theta_s] \rangle \in G_s(\HH{M}({\rm A},{\FF}))$, where the Grassmanian $G_s(\HH{M}({\rm A},{\FF}))$ is the set of $s$-dimensional linear subspaces of $\HH{M}({\rm A},{\FF})$.
Finally (see \cite[Lemma 15]{hac16}), if $\langle \left[ \theta _{1}\right],\dots,
\left[ \theta _{s}\right]\rangle=\langle \left[ \vartheta
_{1}\right],\dots,\left[ \vartheta _{s}\right]
\rangle \in G_s(\HH{M}({\rm A},{\FF}))$, we have that $ \bigcap\limits_{i=1}^{s}\operatorname{Ann}(\theta _{i})\cap \operatorname{Ann}\left( {\rm A}\right) = \bigcap\limits_{i=1}^{s}
\operatorname{Ann}(\vartheta _{i})\cap\operatorname{Ann}( {\rm A}) $, and therefore the set 
\[
T_{s}({\rm A}) =\{ W=\left\langle \left[ \theta _{1}\right] ,\dots,\left[ \theta _{s}\right] \right\rangle \in
G_s(\HH{M}({\rm A},{\FF})) \mid \bigcap\limits_{i=1}^{s}\ann(\theta _{i})\cap\ann({\rm A}) =0\}
\]
is well-defined. Then, the set of all non-split central extensions of ${\rm A}$ with $s$-dimensional annihilator, by some $s$-dimensional vector space ${\rm V}$, can be written as 
\[
E( {\rm A},{\rm V}) =\{ {\rm A}_{\theta }\mid\theta( x,y) = \sum_{i=1}^{s}\theta _{i}( x,y) e_{i} \ \ \text{and} \ \ \langle \left[ \theta _{1}\right],\dots,
\left[ \theta _{s}\right] \rangle \in T_{s}({\rm A})\} .
\]

Now, if $\phi$ is an automorphism of ${\rm A}$ and $W=\langle
\left[ \theta _{1}\right],\dots,\left[ \theta _{s}
\right]\rangle\in G_s(\HH{M}({\rm A},{\FF}))$ it holds that $\phi W=\langle
\left[ \phi\theta _{1}\right],\dots,\left[ \phi\theta _{s}
\right]\rangle$ also has dimension $s$. This induces an action of $\Aut{{\rm A}}$ on $G_s(\HH{M}({\rm A},{\FF}))$; the orbit of $W$ will be denoted by $\orb{(W)}$. The set $T_{s}({\rm A})$ is stable under this action (\cite[Lemma 16]{hac16}).

Having established these results, we can determine whether two $s$-dimensional non-split central extensions ${\rm A}_{\theta}$ and ${\rm A}_{\vartheta}$ in the variety ${\rm M}$ are isomorphic or not. For the proof, see~\cite[Lemma 17]{hac16}.

\begin{lemma}
 Let ${\rm A}_{\theta },{\rm A}_{\vartheta }\in E( {\rm A},{\rm V}) $. Suppose that $\theta ( x,y) =  \displaystyle \sum_{i=1}^{s}
\theta _{i}( x,y) e_{i}$ and $\vartheta ( x,y) =
\displaystyle \sum_{i=1}^{s} \vartheta _{i}( x,y) e_{i}$.
Then the algebras ${\rm A}_{\theta }$ and ${\rm A}_{\vartheta } $ of the variety ${M}$ are isomorphic
if and only if
\[\orb(\langle \left[ \theta _{1}\right] ,
\dots,\left[ \theta _{s}\right] \rangle) =
\orb (\langle \left[ \vartheta _{1}\right] ,\dots,\left[ \vartheta _{s}\right]\rangle).\]
\end{lemma}

Therefore, given an algebra ${\rm A}$ of dimension $n-s$ in the variety ${\rm M}$, we can construct all non-split central extensions ${\rm A}_{\theta}$ of ${\rm A}$ according to the following procedure:

\begin{enumerate}
\item Determine $\HH{M}({\rm A},{\FF})$, ${\rm Ann}({\rm A})$ and $\Aut{\rm A}$.

\item Determine the set of ($\Aut{\rm A}$)-orbits on $T_{s}({\rm A}) $.

\item For each orbit, construct the algebra of the variety ${\rm M}$ associated with a
representative of it.
\end{enumerate}

Then, thanks to this procedure and to Lemma~\ref{lm:first}, we can classify all the nilpotent algebras of dimension $n$ of a variety ${\rm M}$, provided that the nilpotent algebras of dimension $n-1$ are known.

Let ${\rm A}$ be an  algebra and fix a basis $\{e_{1},e_{2},\dots,e_{n}\}$ of ${\rm A}$. We define the bilinear form
$\Delta _{i,j} \colon {\rm A}\times {\rm A}\longrightarrow \mathbb F$
by $\Delta_{i,j}\left( e_{l},e_{m}\right) = \delta_{i,l}\delta_{j,m}$.
Then the set $\left\{ \Delta_{i,j}\mid i, j\in [n]\right\} $ is a basis for the linear space of bilinear forms on ${\rm A}$; in particular, every $\theta \in
{\rm Z^{2}}\left( {\rm A},{\bf V}\right) $ can be uniquely written as $
\theta = \sum_{ i,j=1}^n c_{i,j}\Delta _{i,j}$, with $
c_{i,j}\in \mathbb F$.

\section{Null-filiform associative algebras}\label{S:nfaa}

For an algebra ${\rm A}$ of an arbitrary variety ${\rm M}$, we consider the series
\[{\rm A}^1={\rm A}, \qquad \ {\rm A}^{i+1}=\sum\limits_{k=1}^{i}{\rm A}^k {\rm A}^{i+1-k}, \qquad i\geq 1.\]
An $n$-dimensional algebra ${\rm A}$ is called {\it null-filiform} if $\dim {\rm A}^i=(n+ 1)-i,  \ i\in [n+1].$
All null-filiform associative algebras were described in  \cite{karel}.

\begin{thm}\label{T:nfaa} An arbitrary $n$-dimensional null-filiform associative algebra is isomorphic to the algebra:
\[\mu_0^n : \quad e_i e_j= e_{i+j}, \quad i, j\in [n],\]
where $\{ e_1, e_2, \dots, e_n\}$ is a basis of the algebra $\mu_0^n$ and we define $e_m=0$ for all $m>n$.
\end{thm}

Using the procedure explained in Section~\ref{S:general}, we can easy find all associative central extensions of $\mu_0^n$.
Let $\nabla_ j = \sum\limits_{k=1}^{j}\Delta_{k,j+1-k}$, for $j\in[n]$.
We need the following result from~\cite{kkl18}.
\begin{prop} Let $\mu_0^n$ be the null-filiform associative algebra of dimension $n$. Then
\[\Z{AS}(\mu_0^n,\FF)=\langle \nabla_j \mid  j\in [n]\rangle,\quad 
\B{}(\mu_0^n,\FF)=\langle\nabla_j \mid j\in [n-1]\rangle,\]
\[\HH{AS}(\mu_0^n,\FF)=\Z{A}(\mu_0^n,\FF)/\B{}(\mu_0^n,\FF)=\langle [ \nabla_{n}] \rangle.\]
\end{prop}

\begin{remark}
The above result appears in~\cite{kkl18} for $\FF=\mathbb{C}$, the field of complex numbers. Nevertheless, it is immediate to check, e.g.\ using the methods appearing shortly in this paper, that it does hold over an arbitrary field. 
\end{remark}

As $\dim \HH{AS}(\mu_0^n,\FF)=1$, the next result follows easily.

\begin{thm} Every non-split  $1$-dimensional associative central extension of $\mu_0^n$ is isomorphic to $\mu_0^{n+1}.$
\end{thm}

\smallskip

In the following sections we will study non-associative central extensions of $\mu_0^n$. It is easy to see that $\mu_0^{n+1}$ is an associative  central  extension of $\mu_0^n.$ All extensions of this type will be called trivial. The basis $\{e_1,\dots,e_n\}$ of $\mu_0^n$ will always be assumed to satisfy the relations from Theorem~\ref{T:nfaa}.

\section{Alternative and left alternative central extensions}

Throughout this section, we assume that the characteristic of the filed $\FF$ satisfies $\chara(\FF)\neq 2$. Recall that an algebra ${\rm A}$ is said to be left alternative (respectively, right alternative) if it satisfies the identity 
\[ x(xy)=(xx)y \quad (\mbox{respectively, }(xy)y=x(yy)),\] 
for all $x,y\in{\rm A}$. Also, if ${\rm A}$ is both left alternative and right alternative, it is called alternative. 


Let us consider $\mu_0^n$ as a left alternative algebra. 
Note also that the linearization of the left alternative identity for $\mu_0^n$ leads to 
\[e_i(e_je_k)+e_j(e_ie_k)=(e_ie_j)e_k+(e_je_i)e_k 
,\] 
for $i,j,k\in[n]$.
Then, its space of cocycles is formed by all the bilinear maps $\theta\colon\mu_0^n\times \mu_0^n\to \FF$ satisfying
$\theta(e_i,e_je_k)+\theta(e_j,e_ie_k)=\theta(e_ie_j,e_k)+\theta(e_je_i,e_k).$ 
This can be expressed as
\begin{equation}\label{lalt}
    \theta(e_i,e_{j+k})+\theta(e_j,e_{i+k})=2\theta(e_{i+j},e_k),
\end{equation}
for $i,j,k\in [ n],$ and considering that $e_m=0$ for $m>n$.


\begin{thm}
Assume that $\chara(\FF)\neq 2$. Then all left alternative and all alternative central extensions of $\mu_0^n$ are trivial.
\end{thm}

\begin{proof}
The first step is to compute $\Z{LA}(\mu_0^n, \FF)$. 
Let $\theta=\sum_{i,j}c_{i,j}\Delta_{i,j}$ be an arbitrary cocycle of $\mu_0^n$ considered as a left alternative algebra. The identity~\eqref{lalt} leads to
\begin{equation*}
    c_{i,j+k}+c_{j,i+k}=2c_{i+j,k}
\end{equation*}
for $i,j,k\in  [n]$, with the assumption that $c_{i,j}=0$ if $i>n$ or $j>n$. Given integers $m, s$ such that $m, s-m\in[n]$ and $m\geq 2$, and taking $i=m-1$, $j=1$ and $k=s-m$ in the above equation, we get 
\begin{equation}\label{lartc}
    2c_{m,s-m}= c_{m-1,s-(m-1)}+c_{1,s-1}.
\end{equation}

\textit{Claim:} For all $i, j\in [n]$, $c_{i,j}=c_{1, i+j-1}$. In particular, $c_{i, j}=0$ if $i+j\geq n+2$.

\medskip

The claim will follow by induction on $i$, the case $i=1$ being trivial. So assume that $i\geq 2$ and that the claim holds for $i-1$. Taking $m=i$ and $s=i+j$ in~\eqref{lartc}, we get
\begin{equation*}
2c_{i,j}= c_{i-1,j+1}+c_{1,i+j-1}. 
\end{equation*}
If $j=n$, then $j+1, i+j-1\geq n+1$ so $c_{i-1,j+1}=0=c_{1,i+j-1}$ and it follows from the identity above that $c_{i,j}=0$, as $\chara(\FF)\neq 2$. In particular, $c_{i,j}=c_{1, i+j-1}$ holds. Otherwise, if $j<n$, then we can use the induction hypothesis to obtain $2c_{i,j}= c_{i-1,j+1}+c_{1,i+j-1}=2c_{1,i+j-1}$, whence the claim.

So, it is clear that $\Z{LA}(\mu_0^n, \FF)$ is in the linear space spanned by $\{\nabla_j \}_{j=1}^{n}$; it is also immediate to see that every element from $\Z{AS}(\mu_0^n, \FF)$ is a left alternative cocycle.
Thus,
$\HH{LA}(\mu_0^n, \FF)=\HH{AS}(\mu_0^n, \FF)$, and 
every left alternative central extension is trivial.
Since any alternative extension is left alternative, it follows that every alternative central extension is trivial as well.




\end{proof}

\section{Jordan central extensions}

Throughout this section, we assume that the characteristic of the filed $\FF$ satisfies $\chara(\FF)\neq 2, 3$.
Recall that a commutative algebra $\rm A$ is said to be Jordan if it satisfies \[x^2(yx)=(x^2y)x,\]
for all $x,y\in{\rm A}$. 
It is immediate to check that the commutative algebra $\mu_0^n$ satisfies the previous identity and therefore is a Jordan algebra. 

The space of Jordan cocycles of $\mu_0^n$ is formed by all the bilinear maps $\theta\colon\mu_0^n\times \mu_0^n\to \FF$ satisfying
\begin{equation}\theta(x,y)=\theta(y,x)\label{cond1},\end{equation}
\begin{equation}\theta(x^2, yx)=\theta(x^2y, x)\label{cond2}.\end{equation}
Equivalently, since $\chara(\FF)\neq 2, 3$, we have
\[\theta(e_i,e_j)=\theta(e_j,e_i), \]
\[\theta(e_{i+\ell},e_{j+k})+\theta(e_{j+\ell},e_{i+k})+\theta(e_{k+\ell},e_{i+j})=\theta(e_{i},e_{j+k+\ell})+\theta(e_{j},e_{i+k+\ell})+\theta(e_{k},e_{i+j+\ell}),\]
for $i,j,k,\ell\in [n].$ In particular, taking 
$i=j$, we see that every cocycle must verify
\begin{equation}\label{relB}
2\theta(e_{i+\ell},e_{i+k})+\theta(e_{\ell+k},e_{2i})
= 2\theta(e_{i},e_{i+k+\ell})+\theta(e_{k},e_{2i+\ell}).
\end{equation}

\begin{thm}
All Jordan central extensions of $\mu_0^n$ are trivial.
\end{thm}

\begin{proof}
We will prove that $\{\nabla_j\}_{j=1}^{n}$ is a basis for $\Z{J}(\mu_0^n, \FF)$.

Let $\theta=\sum_{i,j}c_{i,j}\Delta_{i,j}$ be an arbitrary cocycle in $\Z{J}(\mu_0^n, \FF)$. Note that $c_{i,j}=c_{j,i}$, for all $i,j\in[n]$. By \eqref{relB}, we also have
\begin{equation}\label{relcij}
c_{k,2a+b}=2c_{a+b,a+k}+c_{b+k,2a}-2c_{a,a+b+k},
\end{equation}
for all $a, b, k\in [n]$, with the assumption that $c_{i,j}=0$ in case $i>n$ or $j>n$.

\medskip

\textit{Claim:} For all $i, j\in [n]$, $c_{i,j}=c_{1, i+j-1}$. In particular, $c_{i, j}=0$ if $i+j\geq n+2$.

\medskip

The proof is by induction on $i$. If either $i=1$ or $j=1$, the claim is trivial, by commutativity. So we can assume that $i, j\geq 2$. If $i=2$, then
\begin{equation*}
c_{2, j}=\theta(e_1^2, e_{j-1}e_1)=\theta(e_{j+1}, e_1)=\theta(e_1, e_{j+1})=c_{1, j+1}. 
\end{equation*}
Let us assume thus that $i\geq 3$ and $j\geq 2$. There are unique integers $a, b$ with $b\in\pb{1, 2}$ such that $i=2a+b$. Notice that $1\leq a<a+b<i\leq n$. Then, using \eqref{relcij} with $k=j$, we get
\begin{equation*}
c_{i,j}= c_{2a+b,j}=2c_{a+b,a+j}+c_{b+j,2a}-2c_{a,a+b+j}.
\end{equation*}
There are two cases to consider.

\textbf{Case1:} $a+b+j\leq n$. In this case, we can use the inductive hypothesis and the commutativity to get
\begin{align*}
2c_{a+b,a+j}+c_{b+j,2a}-2c_{a,a+b+j} &= 2c_{1,2a+b+j-1}+c_{1,2a+b+j-1}-2c_{1,2a+b+j-1}\\ &=c_{1,2a+b+j-1}=c_{1, i+j-1}.
\end{align*}

\textbf{Case2:} $a+b+j\geq n+1$. Then $c_{a,a+b+j}=0$ and we have $c_{i,j}=2c_{a+b,a+j}+c_{2a,b+j}$. Either $a+j\geq n+1$ and hence $c_{a+b,a+j}=0$, or $a+j\leq n$ and the inductive hypothesis says that $c_{a+b,a+j}=c_{1,2a+b+j-1}=0$, because $2a+b+j-1\geq a+b+j\geq n+1$. In any case, $c_{a+b,a+j}=0$. Similarly, $c_{2a,b+j}=0$ and we conclude that $c_{i,j}=0=c_{1, i+j-1}$, because $i+j-1\geq a+b+j\geq n+1$.

The claim is thus established. It remains to observe that the $\{\nabla_j\}_{j=1}^{n}$ are indeed Jordan cocycles. In~\cite{kkl18}, it is proved that $\nabla_j\in\Z{AS}(\mu_0^n, \FF)$, i.e.\ they verify $\nabla_j(xy,z)=\nabla_j(x,yz)$. Clearly, also $\nabla_j(x,y)=\nabla_j(y,x)$. These conditions are indeed stricter than conditions~\eqref{cond1} and~\eqref{cond2}, so it is clear that $\nabla_j\in\Z{J}(\mu_0^n, \FF)$. This implies that $\HH{J}(\mu_0^n, \FF)=\HH{AS}(\mu_0^n, \FF)$ and, according to~\cite{kkl18}, the only Jordan central extension of $\mu_0^n$ is $\mu_0^{n+1}$.
\end{proof}

\section{Left commutative and bicommutative central extensions}\label{S:lbc}

Recall that an algebra ${\rm A}$ is said to be left (respectively, right) commutative if it satisfies
\[x(yz)=y(xz) \quad \mbox{(respectively, $(xy)z=(xz)y$}),\] 
for all $x,y,z\in{\rm A}$. Equivalently, ${\rm A}$ is left (respectively, right) commutative if and only if the left (respectively, right) multiplication operators commute. In case ${\rm A}$ is both left and right commutative, we say that ${\rm A}$ is bicommutative. 

The main results in this section will require the field $\FF$ to be algebraically closed and of sufficiently large characteristic (which for simplicity we will assume to be $0$, when necessary), while others hold for arbitrary fields. Unless otherwise is stated, it should be assumed that the field $\FF$ is arbitrary.

\begin{prop}\label{P:hh:lbc}
Let $n\geq 2$ and recall the bilinear forms $\nabla_ j = \sum\limits_{k=1}^{j}\Delta_{k,j+1-k}$, defined for $j\in[n]$.
Then the following hold:
\begin{enumerate}[label=\textup{(\alph*)}]
\item $\dim\Z{LC}(\mu_0^n, \FF)=2n-1$ and $\pb{\Delta_{i,1}\mid 2\leq i\leq n}\cup\pb{\nabla_j\mid j\in [n]}$ is a basis of $\Z{LC}(\mu_0^n, \FF)$.\label{P:hh:lbc:a}
\item $\dim\HH{LC}(\mu_0^n, \FF)=n$ and the classes $\pb{\lb{\Delta_{i,1}}\mid 2\leq i\leq n}\cup\pb{\lb{\nabla_n}}$ form a basis of $\HH{LC}(\mu_0^n, \FF)$.\label{P:hh:lbc:b}
\item In the bicommutative case we have that $\dim\HH{BC}(\mu_0^n, \FF)=2$, with basis $\pb{\lb{\Delta_{2, 1}}, \lb{\nabla_n}}$.\label{P:hh:lbc:c}
\end{enumerate}
\end{prop}

\begin{proof}
The space $\Z{LC}(\mu_0^n, \FF)$ consists of the bilinear forms $\theta=\sum_{i, j=1}^n c_{i,j}\Delta_{i,j}$, with $c_{i,j}\in\FF$, satisfying $\theta(e_i, e_j e_k)=\theta(e_j, e_i e_k)$, for all $i, j, k\in[n]$.

\medskip

\textit{Claim:} If $j\geq 2$, then $c_{i,j}=c_{1, i+j-1}$. In particular, if $i+j\geq n+2$ then $c_{i,j}=0$.

\medskip

For $j\geq 2$ we have
\begin{equation}\label{E:P:hh:lbc:theta}
c_{i,j}=\theta(e_i, e_j)=\theta(e_i, e_1 e_{j-1})=\theta(e_1, e_i e_{j-1})=c_{1, i+j-1}. 
\end{equation}
Thus, if $i+j\geq n+2$, then necessarily $j\geq 2$ and \eqref{E:P:hh:lbc:theta} gives $c_{i,j}=0$. This establishes the claim.

\medskip

Hence, $c_{i,2}=c_{i-1, 3}=\cdots=c_{2, i}=c_{1, i+1}$ for all $i\in[n-1]$ and we can write
\begin{equation}\label{E:P:hh:lbc:thetab}
\theta=\sum_{i=1}^n c_{i,1}\Delta_{i,1}+\sum_{j=2}^n c_{1, j}(\nabla_j-\Delta_{j,1})=
\sum_{i=2}^n (c_{i,1}-c_{1, i})\Delta_{i,1}+\sum_{j=1}^n c_{1, j}\nabla_j,
\end{equation}
which shows that $\Z{LC}(\mu_0^n, \FF)\subseteq\langle\pb{\Delta_{i,1}\mid 2\leq i\leq n}\cup\pb{\nabla_j\mid j\in [n]}\rangle$.

It remains to prove the reverse inclusion, i.e.\ that the $\Delta_{i,1}$ and the $\nabla_j$ are left commutative cocycles. Let $k, \ell, m\in[n]$. Then
\begin{equation*}
\Delta_{i, 1}(e_k, e_{\ell} e_m)=0= \Delta_{i, 1}(e_{\ell}, e_k e_m),
\end{equation*}
so indeed $\Delta_{i, 1}\in\Z{LC}(\mu_0^n, \FF)$. As $\langle\nabla_j \mid j\in [n-1]\rangle=\B{}(\mu_0^n, \FF)\subseteq \Z{LC}(\mu_0^n, \FF)$, it remains to show that $\nabla_n$ is a cocycle. This follows immediately from the fact that 
\begin{equation*}\nabla_n(e_k, e_{\ell}e_m)=
\begin{cases}
1&\mbox{if $k+\ell+m=n+1$;}\\
0&\mbox{otherwise.} 
\end{cases} 
\end{equation*}
This concludes the proof of \ref{P:hh:lbc:a}. Part \ref{P:hh:lbc:b} is clear by $\B{}(\mu_0^n, \FF)=\langle\nabla_j \mid j\in [n-1]\rangle$.

We precede the proof of \ref{P:hh:lbc:c} with some general considerations. Let ${\rm A}$ be an algebra and denote by $\nu:{\rm A}\times {\rm A}\longrightarrow {\rm A}$ its multiplication map. The opposite algebra ${\rm A}^{\mathrm{op}}$ is the algebra with the same underlying vector space and with multiplication $\nu\circ \tau$, where $\tau:{\rm A}\times {\rm A}\longrightarrow {\rm A}\times {\rm A}$, $(a, b)\mapsto(b,a)$ is the flip. Clearly, ${\rm A}$ is right commutative if and only if ${\rm A}^{\mathrm{op}}$ is left commutative and $\theta\in\Z{RC}({\rm A}, \FF)$ if and only if $\theta\circ\tau\in\Z{LC}({\rm A}^{\mathrm{op}}, \FF)$. Moreover, $\Z{BC}({\rm A}, \FF)=\Z{LC}({\rm A}, \FF)\cap \Z{RC}({\rm A}, \FF)$.

We return now to the case of the algebra $\mu_0^n$. This algebra is commutative, so it coincides with its opposite algebra. Let $\theta\in\Z{BC}(\mu_0^n, \FF)$. Then $\theta\in\Z{LC}(\mu_0^n,\FF)$ and by \ref{P:hh:lbc:a} we can write $\theta$ as in \eqref{E:P:hh:lbc:thetab}. What is more, for such a $\theta$ we have $\theta\in\Z{BC}(\mu_0^n, \FF)$ if and only if $\theta\circ\tau\in\Z{LC}(\mu_0^n, \FF)$. Seeing that $\Delta_{i,j}\circ\tau=\Delta_{j,i}$, we obtain $\nabla_j\circ\tau=\nabla_{j}$, for all $j\in [n]$.

Thence, in view of the results in \ref{P:hh:lbc:a}, and the observation that $\Delta_{1, 2}=\nabla_2-\Delta_{2, 1}$, we have that $\theta\circ\tau\in\Z{LC}(\mu_0^n, \FF)$ if and only if $c_{i,1}=c_{1, i}$ for all $3\leq i\leq n$.
This proves that $\pb{\nabla_j\mid  j\in [n]}\cup\pb{\Delta_{2, 1}}$ is a basis of $\Z{BC}(\mu_0^n, \FF)$. Finally, given our description of $\B{}(\mu_0^n, \FF)$, we immediately obtain the basis $\pb{\lb{\Delta_{2, 1}}, \lb{\nabla_n}}$ of $\HH{BC}(\mu_0^n, \FF)$.
\end{proof}

\subsection{The automorphism group of $\mu_0^n$}\label{SS:action:auto}

We denote by $\Aut (\mu_0^n)$ the automorphism group of $\mu_0^n$. Let $\phi\in\Aut (\mu_0^n)$. Then we identify $\phi$ with its matrix $\seq{\phi_{i,j}}_{ i, j\in [n]}$ relative to the basis $\pb{e_1,\dots,e_n}$. It is easy to see that the automorphisms of $\mu_0^n$ are precisely those linear endomorphisms $\phi=\seq{\phi_{i,j}}_{ i, j\in [n]}$ with $\phi_{1,1}\neq 0$, $\phi_{2, 1}, \ldots, \phi_{n,1}\in\FF$ arbitrary and $\phi_{i,j}$ with $2\leq j\leq n$ determined by $\phi(e_j)=\phi(e_1)^j$. In other words,
\begin{equation*}\label{E:aut:def}
    \phi_{i,j}=\sum_{k_1+\cdots+k_j=i}\phi_{k_1, 1}\cdots \phi_{k_j, 1},
\end{equation*}
for all $i, j\in[n]$.

It follows that $\phi_{i,i}=\phi_{1,1}^i$ and $\phi_{i,j}=0$ if $j>i$. For $j<i$ we also have
\begin{equation}\label{E:aut:fij}
    \phi_{i,j}=j\phi_{1, 1}^{j-1}\phi_{i-j+1, 1}+p_{i,j}(\phi_{1,1}, \phi_{2, 1}, \ldots, \phi_{i-j, 1}),
\end{equation}
for some polynomial $p_{i,j}$ with coefficients in $\FF$ which depends only on $i$ and $j$.

Write $\phi=\sum_{i, j=1}^n\phi_{i,j}E_{i,j}$, where we think of $E_{i,j}$ both as the matrix unit with all entries equal to $0$ except for the entry $(i, j)$ which equals $1$, and also as the linear endomorphism of $\mu_0^n$ defined by $E_{i,j}(e_k)=\delta_{j,k}e_i$. Then,
\begin{align*}
    \phi\cdot\Delta_{s,t}(e_k, e_\ell)&=\Delta_{s,t}(\phi(e_k), \phi(e_\ell))
    =\sum_{k\leq i}\sum_{\ell\leq j}\phi_{i,k}\phi_{j,\ell}\Delta_{s,t}(e_i, e_j)=\sum_{k\leq s,\ \ell\leq t}\phi_{s,k}\phi_{t,\ell}.
\end{align*}
Thence, the formula for the action is given as
\begin{equation}\label{E:auto:formula}
    \phi\cdot\Delta_{s,t}=\sum_{k\leq s,\ \ell\leq t}\phi_{s,k}\phi_{t,\ell}\Delta_{k, \ell}.
\end{equation}

\subsection{The orbit decomposition of $\HH{LC}(\mu_0^n, \FF)$}\label{SS:action:orbits:lcom}


To state the main result of this subsection we need an additional definition. For $0\leq i\leq n$, let $R(i, n)$ be the multiplicative subgroup of $\FF^*$ consisting of all $(i+1)$-st powers of all $(n+1)$-st roots of unity in $\FF^*$. In case $\FF$ is algebraically closed of characteristic $0$, there exists some primitive $(n+1)$-st root of unity $\zeta$ and $R(i, n)=\langle\zeta^{i+1}\rangle$ is the cyclic group generated by $\zeta^{i+1}$. Denote the quotient group by
\begin{equation}\label{E:def:fin}
\FF_{(i, n)} = \FF^*/R(i, n),
\end{equation}
and for $\mu\in\FF^*$, let $\overline{\mu}=\mu R(i, n)$ be the corresponding coset.

\begin{thm}\label{T:action:orbits:lcom}
Assume that $\FF$ is algebraically closed of characteristic $0$ and let $n\geq 2$. The following elements of $\Z{LC}(\mu_0^n, \FF)$ give a complete list of distinct representatives of the orbits of the automorphism group $\Aut (\mu_0^n)$ on $\HH{LC}(\mu_0^n, \FF):$
\begin{enumerate}[label=\textup{(\alph*)}]
\item $0;$
\item $\pb{\Delta_{i, 1}\mid 2\leq i\leq n};$
\item $\pb{\nabla_n+\mu\Delta_{n, 1}\mid \mu\in\FF};$
\item $\pb{\nabla_n+\overline{\mu}\Delta_{i, 1}\mid 2\leq i\leq n-1,\ \overline{\mu}\in\FF_{(i, n)}}.$
\end{enumerate}
\end{thm}

From this result and Proposition~\ref{P:hh:lbc}, we immediately deduce the orbit space decomposition in the bicommutative case.

\begin{cor}\label{C:action:orbits:bicom}
Assume that $\FF$ is algebraically closed of characteristic $0$ and let $n\geq 2$. The following elements of $\Z{BC}(\mu_0^n, \FF)$ give a complete list of distinct representatives of the orbits of the automorphism group $\Aut (\mu_0^n)$ on $\HH{BC}(\mu_0^n, \FF):$ 
\begin{description}
\item[(Case $n=2$)]\hfill
\begin{enumerate}[label=\textup{(\alph*)}]
\item $0;$
\item $\Delta_{2, 1};$
\item $\pb{\nabla_2+\mu\Delta_{2, 1}\mid \mu\in\FF}.$
\end{enumerate}
\item[(Case $n>2$)]\hfill
\begin{enumerate}[label=\textup{(\alph*)}]
\item $0;$
\item $\Delta_{2, 1};$
\item $\nabla_n;$
\item $\pb{\nabla_n+\overline{\mu}\Delta_{2, 1}\mid \overline{\mu}\in\FF_{(2, n)}}.$
\end{enumerate}
\end{description}
\end{cor}

We devote the remainder of this section to the proof of Theorem~\ref{T:action:orbits:lcom}.

\begin{lemma}\label{L:action:arith}
Assume that $\FF$ is algebraically closed of characteristic $0$. Fix $i\geq 2$. Given scalars $a_{1, 1}, \ldots, a_{i-1, 1}\in\FF$ with $a_{1, 1}\neq 0$, we define, for $2\leq \ell, m\leq i:$
\begin{equation}\label{E:L:action:arith}
a_{\ell, m}=\sum_{j_1+\cdots+j_m=\ell}a_{j_1, 1}\cdots a_{j_m, 1}.
\end{equation}
Note that $a_{\ell, m}$ depends only on $a_{1, 1}, \ldots, a_{\ell-m+1, 1}$. Then, for any $0\leq k\leq i-1$, the map
\begin{equation*}
\begin{array}{rccc}
\rho_k:&\FF^*\times\FF^k&\longrightarrow& \FF^*\times\FF^k\\
&\seq{a_{1, 1}, \ldots, a_{k+1, 1}}&\mapsto&\seq{a_{i,i}a_{1,1}, a_{i, i-1}a_{1, 1}, \ldots, a_{i, i-k}a_{1,1}}
\end{array}
\end{equation*}
is onto and $(i+1)$-to-$1$, i.e.\ $\left| \rho_k^{-1}(\lambda_1, \ldots, \lambda_{k+1})\right|=i+1$, for all $(\lambda_1, \ldots, \lambda_{k+1})\in \FF^*\times\FF^k$.
\end{lemma}

\begin{proof}
The proof is by induction on $0\leq k\leq i-1$. First, notice that \eqref{E:L:action:arith} implies that $a_{\ell, \ell}=a_{1, 1}^\ell$. 
If $k=0$, then $\rho_0(a_{1, 1})=a_{1, 1}^{i+1}$ and the relation $\left| \rho_0^{-1}(\lambda
)\right|=i+1$ for any $\lambda\in\FF^*$ follows from both our assumptions on $\FF$.

Now assume that $k\geq 1$ and take $(\lambda_1, \ldots, \lambda_{k+1})\in \FF^*\times\FF^k$. By the induction hypothesis, $\left| \rho_{k-1}^{-1}(\lambda_1, \ldots, \lambda_{k})\right|=i+1$, so choose $(a_{1, 1}, \ldots, a_{k, 1})\in\FF^*\times\FF^{k-1}$ such that $\rho_{k-1}(a_{1, 1}, \ldots, a_{k, 1})=(\lambda_1, \ldots, \lambda_{k})$. As in \eqref{E:aut:fij}, we have $a_{i, i-k}a_{1,1}=(i-k)a_{1, 1}^{i-k}a_{k+1, 1}+a_{1, 1}p_{i, i-k}\seq{a_{1, 1}, \ldots, a_{k, 1}}$, with $p_{i, i-k}$ a polynomial in $a_{1, 1}, \ldots, a_{k, 1}$. Thus, we can set
\begin{equation}\label{E:akp11}
a_{k+1, 1}=\frac{\lambda_{k+1}-a_{1, 1}p_{i, i-k}\seq{a_{1, 1}, \ldots, a_{k, 1}}}{(i-k)a_{1, 1}^{i-k}},
\end{equation}
and, by construction, $\rho_{k}(a_{1, 1}, \ldots, a_{k+1, 1})=\seq{\rho_{k-1}(a_{1, 1}, \ldots, a_{k, 1}), a_{i, i-k}a_{1,1}}=(\lambda_1, \ldots, \lambda_{k+1})$.

The above shows that the $i+1$ distinct solutions of the problem for $k-1$ give rise to $i+1$ distinct solutions of the problem for $k$. Conversely, each solution $(a_{1, 1}, \ldots, a_{k+1, 1})$ of the latter determines a solution $(a_{1, 1}, \ldots, a_{k, 1})$ of the former and, by \eqref{E:akp11}, $a_{k+1, 1}$ is completely determined by $a_{1, 1}, \ldots, a_{k, 1}$, so there are exactly $i+1$ solutions of the problem for $k$. By induction, the proof is complete.
\end{proof}

Now we can start computing orbits of $\Aut (\mu_0^n)$ on $\Z{LC}(\mu_0^n, \FF)$.

\begin{prop}\label{T:action:orbits:lcom:1}
Assume that $\FF$ is algebraically closed of characteristic $0$. For $i\in [n]$ we have
\begin{equation*}
\orb\seq{\Delta_{i, 1}}=\pb{\lambda_1\Delta_{i, 1}+\lambda_2 \Delta_{i-1, 1}+\cdots +\lambda_i \Delta_{1, 1}\mid \lambda_1, \ldots, \lambda_i\in\FF, \lambda_1\neq 0}. 
\end{equation*}
\end{prop}
\begin{proof}
Let $\phi=\seq{\phi_{i,j}}\in\Aut (\mu_0^n).$ Then, by \eqref{E:auto:formula}, $\phi\cdot\Delta_{i, 1}=\phi_{1,1}^{i+1}\Delta_{i, 1} + \sum_{k=1}^{i-1}\phi_{i,k}\phi_{1,1}\Delta_{k, 1}$ and as $\phi_{1, 1}\neq 0$, the direct inclusion in the statement is proved.

Conversely, given $(\lambda_1, \ldots, \lambda_{i})\in \FF^*\times\FF^{i-1}$, Lemma~\ref{L:action:arith} gives $(a_{1, 1}, \ldots, a_{i, 1})\in\FF^*\times\FF^{i-1}$ such that $\lambda_{i-k+1}=a_{i,k}a_{1, 1}$, for all $k\in[i]$. Since $a_{1, 1}\neq 0$, there exists $\phi\in\Aut (\mu_0^n)$ such that $\phi_{k, 1}=a_{k, 1}$, for all $k\in[i]$. For any such $\phi\in \Aut (\mu_0^n)$ and $j\in[i]$, we have 
\begin{equation*}
a_{i, j}=\sum_{k_1+\cdots+k_j=i}a_{k_1, 1}\cdots a_{k_j, 1}=\sum_{k_1+\cdots+k_j=i}\phi_{k_1, 1}\cdots \phi_{k_j, 1}=\phi_{i, j},
\end{equation*}
so $\phi\cdot\Delta_{i, 1}=\sum_{k=1}^{i}\phi_{i,k}\phi_{1,1}\Delta_{k, 1}=\sum_{k=1}^{i}a_{i,k}a_{1,1}\Delta_{k, 1}=\sum_{k=1}^{i}\lambda_{i-k+1,1}\Delta_{k, 1}$. This proves the reverse inclusion.
\end{proof}

Recall that the $\nabla_{j}$, for $ j\in [n-1]$, form a basis of $\B{}(\mu_0^n, \FF)$.

\begin{lemma}\label{L:phi:nabla}
Let $\phi=\seq{\phi_{i,j}}\in\Aut (\mu_0^n)$. In $\HH{LC}(\mu_0^n, \FF)$ we have $\phi\cdot\lb{\nabla_n}=\phi_{1, 1}^{n+1}\lb{\nabla_n}$.
\end{lemma}
\begin{proof}
If $i+j>n+1$, then $\Delta_{i, j}$ does not occur in $\phi\cdot\nabla_n$. Otherwise, for $i+j\leq n+1$, the coefficient of $\Delta_{i,j}$ in the expression for $\phi\cdot\nabla_n$ is
\begin{equation*}
\alpha_{i,j}=\sum_{k=i}^{n+1-j}\phi_{k,i}\phi_{n+1-k,j}, 
\end{equation*}
so that $\phi\cdot\nabla_n=\sum_{i+j\leq n+1}\alpha_{i,j}\Delta_{i,j}$. 

Since $\Aut (\mu_0^n)$ acts on $\Z{LC}(\mu_0^n, \FF)$ and on $\B{}(\mu_0^n, \FF)$, we must have $\phi\cdot\nabla_n\in\Z{LC}(\mu_0^n, \FF)\setminus\B{}(\mu_0^n, \FF)$. Thus, by the proof of Proposition~\ref{P:hh:lbc}, we have $\alpha_{1, k}=\alpha_{2, k-1}=\cdots=\alpha_{k-1, 2}$ for every $2\leq k\leq n$. Moreover, by reparameterizing the summation index, we obtain $\alpha_{i,j}=\sum_{\ell=j}^{n+1-i}\phi_{n+1-\ell, i}\phi_{\ell, j}=\alpha_{j, i}$; in particular, $\alpha_{k, 1}=\alpha_{1, k}$. It follows that 
\begin{equation*}
\phi\cdot\nabla_n=\sum_{k=1}^n\alpha_{1, k}\nabla_k.
\end{equation*}
Now, as $\lb{\nabla_k}=0$, for all $k\in[n-1]$, we get $\phi\cdot\lb{\nabla_n}=\alpha_{1, n}\lb{\nabla_n}=\phi_{1, 1}^{n+1}\lb{\nabla_n}$. 
\end{proof}

\begin{prop}\label{T:action:orbits:lcom:2}
Assume that $\FF$ is algebraically closed of characteristic $0$. Let $\mu\in \FF^*$. We have
\begin{equation*}
\orb\seq{\lb{\nabla_n+\mu\Delta_{n, 1}}}=\pb{\lambda_1\lb{\nabla_{n}}+\mu\lambda_1\lb{\Delta_{n, 1}}+\sum_{k=2}^{n-1}\lambda_{n+1-k} \lb{\Delta_{k, 1}}\mid \lambda_1, \ldots, \lambda_{n-1}\in\FF, \lambda_1\neq 0}. 
\end{equation*}
\end{prop}
\begin{proof}
Let $\phi=\seq{\phi_{i,j}}\in\Aut (\mu_0^n)$. Then, using Lemma~\ref{L:phi:nabla} and recalling that $\lb{\Delta_{1, 1}}=\lb{\nabla_1}=0$,
\begin{align}\label{E:action:orbits:lcom:2}
\phi\cdot\lb{\nabla_n+\mu\Delta_{n, 1}} 
=\phi_{1, 1}^{n+1}\lb{\nabla_n}+\mu\phi_{n,n}\phi_{1,1}\lb{\Delta_{n, 1}}
+\mu\sum_{k=2}^{n-1}\phi_{n,k}\phi_{1,1}\lb{\Delta_{k, 1}}.
\end{align}
Since $\phi_{n,n}=\phi_{1,1}^n$ and $\phi_{1,1}\neq 0$, the direct inclusion in the statement follows.

Conversely, let $\lambda_1, \ldots, \lambda_{n-1}\in\FF$, with $\lambda_1\neq 0$. Consider the map $\rho_{n-1}$ from Lemma~\ref{L:action:arith}. By that result, for any $\mu\in\FF^*$, there are $\seq{a_{1,1}, \ldots, a_{n, 1}}\in\FF^*\times\FF^{n-1}$ such that $\rho_{n-1}\seq{a_{1,1}, \ldots, a_{n, 1}}=\seq{\lambda_1, \frac{\lambda_2}{\mu}, \ldots, \frac{\lambda_{n-1}}{\mu}, 0}$. Let $\phi\in\Aut (\mu_0^n)$ be determined by $\phi_{k, 1}=a_{k, 1}$, for all $k\in[n]$. Then $\phi_{i,j}=a_{i,j}$ for all $i, j\in[n]$. It follows that $\phi_{1, 1}^{n+1}=\phi_{n,n}\phi_{1, 1}=\lambda_1$ and similarly $\mu\phi_{n,k}\phi_{1, 1}=\lambda_{n+1-k}$ for all $2\leq k\leq n-1$. Thence, by \eqref{E:action:orbits:lcom:2},
\begin{equation*}
\phi\cdot\lb{\nabla_n+\mu\Delta_{n, 1}}= 
\lambda_1\lb{\nabla_n}+\mu\lambda_1\lb{\Delta_{n, 1}}
+\sum_{k=2}^{n-1}\lambda_{n+1-k}\lb{\Delta_{k, 1}},
\end{equation*}
proving the reverse inclusion.
\end{proof}

\begin{prop}\label{T:action:orbits:lcom:3}
Assume that $\FF$ is algebraically closed of characteristic $0$. For $2\leq i\leq n-1$ and $\mu\in \FF^*$ we have
\begin{equation*}
\orb \seq{\lb{\nabla_n+\mu\Delta_{i, 1}}}=\pb{\lambda^{n+1}\lb{\nabla_{n}}+\lambda^{i+1}\mu\lb{\Delta_{i, 1}}+\sum_{j=2}^{i-1}\lambda_{i+1-j} \lb{\Delta_{j, 1}}\mid \lambda_2, \ldots, \lambda_{i-1}\in\FF, \lambda\in\FF^*}. 
\end{equation*}
\end{prop}
\begin{proof}
Let $\phi=\seq{\phi_{i,j}}\in\Aut (\mu_0^n)$. Then, 
\begin{equation*}
\phi\cdot\lb{\nabla_n+\mu\Delta_{i, 1}}=\phi_{1, 1}^{n+1}\lb{\nabla_n}+\sum_{k=2}^i\phi_{i, k}\phi_{1,1}\mu\lb{\Delta_{k, 1}}.
\end{equation*}
So, as before, to prove the result it suffices to assume that $i\geq 3$ and to show that given $\phi_{1, 1}, \mu\in\FF^*$, the elements $\phi_{i, k}\phi_{1,1}\mu$, with $2\leq k\leq i-1$, can take on arbitrary values in $\FF$, for appropriate choices of $\phi_{2, 1}, \ldots, \phi_{i-1, 1}\in\FF$. 

Considering $\rho_{i-2}$, we know that for any $\lambda_2, \ldots, \lambda_{i-1}\in\FF$ there are $i+1$ solutions to the equation
\begin{equation*}
\rho_{i-2}(a_{1,1}, \ldots, a_{i-1, 1})=\seq{\phi_{1, 1}^{i+1}, \frac{\lambda_2}{\mu}, \ldots, \frac{\lambda_{i-1}}{\mu}}.
\end{equation*}
As $a_{1, 1}^{i+1}=\phi_{1, 1}^{i+1}$, exactly one of the above solutions satisfies $a_{1, 1}=\phi_{1, 1}$ and the remainder of the proof goes as before.
\end{proof}

The final ingredient in the proof of Theorem~\ref{T:action:orbits:lcom} explains the relevance of the factor group $\FF_{(i, n)}$, defined in \eqref{E:def:fin}.

\begin{lemma}\label{L:orbit:param:mubar}
Fix $2\leq i\leq n-1$. For $\mu, \mu'\in \FF^*$ we have
\begin{equation*}
\orb \seq{\lb{\nabla_n+\mu\Delta_{i, 1}}}= \orb  \seq{\lb{\nabla_n+\mu'\Delta_{i, 1}}}\iff \overline{\mu}=\overline{\mu'},
\end{equation*}
where $\overline{\mu}, \overline{\mu'}\in\FF_{(i, n)}$.
\end{lemma}

\begin{remark}
In view of this result, it makes sense to write $\orb \seq{\lb{\nabla_n+\overline{\mu}\Delta_{i, 1}}}$, for $\mu\in\FF^*$, and we can parameterize the orbits of the form above by the elements of $\FF_{(i, n)}$.
\end{remark}

\begin{proof}
For the direct implication, suppose that $\lb{\nabla_n+\mu'\Delta_{i, 1}}\in\orb \seq{\lb{\nabla_n+\mu\Delta_{i, 1}}}$. Then there is some $\lambda\in\FF^*$ such that $\lambda^{n+1}=1$ and $\mu'=\lambda^{i+1}\mu$, so $\mu'/\mu\in R(i, n)$.

Conversely, if $\mu'=\lambda^{i+1}\mu$ for some $\lambda\in\FF^*$ with $\lambda^{n+1}=1$, then
\begin{equation*}
\lb{\nabla_n+\mu'\Delta_{i, 1}}
=\lambda^{n+1}\lb{\nabla_n}+\lambda^{i+1}\mu\lb{\Delta_{i, 1}}
\in\orb \seq{\lb{\nabla_n+\mu\Delta_{i, 1}}},
\end{equation*}
which shows that the respective orbits coincide.
\end{proof}

We are now ready to conclude our main result of this section.

\begin{proof}[Proof of Theorem~\ref{T:action:orbits:lcom}]
Let $0\neq\theta=\lambda_1\lb{\nabla_n}+\sum_{j=2}^n \lambda_j \lb{\Delta_{j,1}}\in\HH{LC}(\mu_0^n, \FF)$. Then:
\begin{itemize}
\item If $\lambda_1=0$, then $\theta\in\orb\seq{\lb{\Delta_{i, 1}}}$, where $i=\max\pb{j\mid \lambda_j\neq 0}$. (See Proposition~\ref{T:action:orbits:lcom:1}.)
\item If $\lambda_1\neq0$ and $\lambda_n\neq 0$, then $\theta\in\orb\seq{\lb{\nabla_n+\mu\Delta_{n, 1}}}$, where $\mu=\lambda_n/\lambda_1\neq 0$. (See Proposition~\ref{T:action:orbits:lcom:2}.)
\item If $\lambda_1\neq0$ and $\lambda_j=0$ for all $2\leq j\leq n$, then $\theta\in\orb\seq{\lb{\nabla_n}}$. (See Lemma~\ref{L:phi:nabla}.)
\item If $\lambda_1\neq0$, $\lambda_n=0$ and $\lambda_j\neq0$ for some $2\leq j\leq n-1$, then $\theta\in\orb  \seq{\lb{\nabla_n+\overline{\mu}\Delta_{i, 1}}}$, where $i=\max\pb{2\leq j\leq n-1\mid \lambda_j\neq 0}$ and $\mu\in\FF^*$ is defined as follows: choose some $\lambda\in\FF^*$ such that $\lambda^{n+1}=\lambda_1$; then take $\mu=\lambda_i/\lambda^{i+1}$. (See Proposition~\ref{T:action:orbits:lcom:3}.) By Lemma~\ref{L:orbit:param:mubar}, $\overline{\mu}$ is uniquely determined by $\theta$. 
\end{itemize}
The fact that distinct elements given in the statement of the Theorem yield disjoint orbits follows easily from the description of the orbits in Propositions~\ref{T:action:orbits:lcom:1}, \ref{T:action:orbits:lcom:2}, \ref{T:action:orbits:lcom:3} and Lemmas~\ref{L:phi:nabla}, \ref{L:orbit:param:mubar}.
\end{proof}

\subsection{The orbit decomposition of $T_1(\mu_0^n)$ in the left commutative and bicommutative varieties}\label{SS:action:orbits:grassm}

Observe that $\ann(\mu_0^n)=\FF e_n$, so in order to get a non-split central extension of $\mu_0^n$ with  $1$-dimensional annihilator associated with a cocycle $\theta$, we must have $e_n\notin\ann(\theta)$. This excludes the orbit representatives $0$ and $\Delta_{i, 1}$, for $i<n$. Notice also that for nonzero cocycles $\theta$ and $\theta'$, the $1$-dimensional spaces $\langle\lb{\theta}\rangle$ and $\langle\lb{\theta'}\rangle$ are in the same $\Aut(\mu_0^n)$-orbit in the corresponding Grassmannian if and only if there is $\lambda\in\FF^*$ such that $\lambda\lb{\theta'}$ is in the orbit of $\lb{\theta}$. In particular, if the $\Aut(\mu_0^n)$-orbit of $\lb{\theta}$ is closed under the multiplicative action of $\FF^*$, then $\langle\lb{\theta}\rangle$ and $\langle\lb{\theta'}\rangle$ are in the same orbit if and only if $\lb{\theta}$ and $\lb{\theta'}$ are in the same orbit.

In our context, all orbits are closed under multiplication by nonzero scalars, except for the orbits of the form $\orb \seq{\lb{\nabla_n+\overline{\mu}\Delta_{i, 1}}}$, with $2\leq i\leq n-1$ and $\overline{\mu}\in\FF_{(i, n)}$. Given $\epsilon\in\FF^*$, Proposition~\ref{T:action:orbits:lcom:3} shows that $\epsilon^{n+1}\lb{\nabla_n+\overline{\mu}\Delta_{i, 1}}=\lb{\nabla_n+\overline{\epsilon^{n-i}\mu}\Delta_{i, 1}}$. Since $\FF$ is algebraically closed, for a fixed $i\leq n-1$, $\epsilon^{n-i}$ can take on any nonzero value as $\epsilon\in\FF^*$ varies. Hence, $\lb{\nabla_n+\overline{\mu}\Delta_{i, 1}}$ and $\lb{\nabla_n+\overline{\mu'}\Delta_{j, 1}}$ define $1$-dimensional spaces in the same $\Aut(\mu_0^n)$-orbit of the Grassmannian if and only if $i=j$. 

Combining all the results from Section~\ref{S:lbc}, we obtain our main result below. 

\begin{thm}\label{T:action:orbits:Grass}
Assume that $\FF$ is algebraically closed of characteristic $0$ and let $n\geq 2$. The following elements of $\Z{LC}(\mu_0^n, \FF)$ parameterize the distinct orbits of $\Aut(\mu_0^n)$ on the subspace $T_1(\mu_0^n)$ of the Grassmannian $G_1(\HH{LC}(\mu_0^n,{\FF}))$, i.e.\ they parameterize the distinct isomorphism classes of non-split left commutative central extensions of the $n$-dimensional null-filiform algebra $\mu_0^n$ with $1$-dimensional annihilator:
\begin{enumerate}[label=\textup{(\alph*)}]
\item $\Delta_{n, 1};$
\item $\pb{\nabla_n+\mu\Delta_{n, 1}\mid \mu\in\FF}$ $(\mu=0$ gives the trivial extension $\mu_0^{n+1});$
\item $\pb{\nabla_n+\Delta_{i, 1}\mid 2\leq i\leq n-1}.$
\end{enumerate} 
In the bicommutative case, the corresponding representatives are:

\begin{varwidth}[t]{.5\textwidth}
\textbf{(Case $n=2$)}
\begin{enumerate}[label=\textup{(\alph*)}]
\item $\Delta_{2, 1};$
\item $\pb{\nabla_2+\mu\Delta_{2, 1}\mid \mu\in\FF}$\\ $(\mu=0$ gives the trivial extension $\mu_0^{3}).$
\end{enumerate}
\end{varwidth}%
\hspace{4em}%
\begin{varwidth}[t]{.5\textwidth}
\textbf{(Case $n>2$)}
\begin{enumerate}[label=\textup{(\alph*)}]
\item $\nabla_n$ $($trivial extension $\mu_0^{n+1});$
\item $\nabla_n+\Delta_{2, 1}.$
\end{enumerate}
\end{varwidth}
\end{thm}

\begin{remark}\label{R:upperdimension}
Note that Theorem~\ref{T:action:orbits:Grass} and the previous discussion only make reference to non-split central extensions with $1$-dimensional annihilator. However, it is also possible to construct non-split left commutative central extensions of $\mu_0^n$ with $2$-dimensional annihilator, whose isomorphism classes are parameterized by the cocycles $\Delta_{k,1}$ for $2\leq  k\leq n-1$. In the bicommutative case for $n>2$, also $\Delta_{2,1}$ is a representative of an isomorphism class of non-split central extensions with $2$-dimensional annihilator.
\end{remark}

The explicit description of the multiplication in the central extensions referenced in Theorems~\ref{T:action:orbits:lcom} and \ref{T:action:orbits:Grass}, in Corollary~\ref{C:action:orbits:bicom} and in Remark~\ref{R:upperdimension} can be found in the following table. Only the nonzero products of the basis elements $\pb{e_1, \ldots, e_{n+1}}$ are displayed.

\begin{table}[hbt!]
\begin{center}
\begin{tabular}{l@{\qquad}ll}\hline\\[-10pt]
cocycle & multiplication, $i, j\in [n]$&  \\[3pt] \hline\hline
$\Delta_{n, 1}$& $e_i e_j=e_{i+j}$ & if $i+j\leq n$ \\[3pt] 
& $e_n e_1=e_{n+1}$ & \\[3pt] \hline
$\Delta_{k, 1}$  & $e_i e_j=e_{i+j}$ & if $i+j\leq n$ and $(i, j)\neq (k, 1)$ \\[3pt] 
($2\leq k\leq n-1$)&$e_k e_1=e_{k+1}+e_{n+1}$&\\[3pt]  \hline
$\nabla_n+\Delta_{k, 1}$  & $e_i e_j=e_{i+j}$ & if $i+j\leq n+1$ and $(i, j)\neq (k, 1)$ \\[3pt] 
($2\leq k\leq n-1$)&$e_k e_1=e_{k+1}+e_{n+1}$&\\[3pt]  \hline
$\nabla_n+\mu\Delta_{n, 1}$ & $e_i e_j=e_{i+j}$ & if $i+j\leq n+1$ and $i\neq n$  \\[3pt]  
($\mu\in\FF$)&$e_n e_1=(1+\mu)e_{n+1}$ &  ($\mu=0$ gives the trivial extension)  \\[3pt]\hline
\end{tabular}
\end{center}
\caption{Isomorphism classes of $1$-dimensional non-split left commutative and bicommutative central extensions of the $n$-dimensional null-filiform algebra $\mu_0^n$.}
\label{table1}
\end{table}%


\section{Assosymmetric, Novikov, and left symmetric central extensions}

\subsection{Assosymmetric central extensions}

Recall that an algebra ${\rm A}$ is said to be assosymmetric if it satisfies the identities 
\[(xy)z-x(yz)=(yx)z-y(xz)=(xz)y-x(zy), \]  
for all $x,y,z\in{\rm A}$. Now, to find the assosymmetric central extensions of $\mu_0^n,$ we need cocycles $\theta\colon \mu_0^n\times \mu_0^n\to \FF$  satisfying  
\[ \theta(e_ie_j,e_k)-\theta(e_i,e_je_k)=\theta(e_je_i,e_k)-\theta(e_j,e_ie_k)=\theta(e_ie_k,e_j)-\theta(e_i,e_ke_j),\]
for all $i,j,k \in [n].$
Note that for the algebra $\mu_0^n$, these two equalities reduce to 
\[\theta(e_i,e_{j+k})=\theta(e_j,e_{i+k}) \quad \mbox{and}\quad  \theta(e_{i+j},e_k)=\theta(e_{i+k},e_j),\]
for all $i,j,k \in [n]$, with the usual convention that $e_m=0$ if $m>n$.
This means that every assosymmetric cocycle of $\mu_0^n$ is in $\Z{BC}(\mu_0^n, \FF).$ 
On the other hand, it is easy to see that every element from $\Z{BC}(\mu_0^n, \FF)$ is an assosymmetric cocycle.

\subsection{Novikov central extensions}

Recall that an algebra ${\rm A}$ is said to be Novikov if it satisfies the identities 
\[(xy)z=(xz)y \quad \mbox{and}\quad(xy)z-x(yz)=(yx)z-y(xz),\] 
for all $x,y,z\in{\rm A}$. Now, to find all the Novikov central extensions  of $\mu_0^n,$ we require cocycles $\theta\colon \mu_0^n\times \mu_0^n\to \FF$ satisfying 
\[\theta(e_ie_j,e_k)=\theta(e_ie_k,e_j)\quad\mbox{and}\quad\theta(e_ie_j,e_k)-\theta(e_i,e_je_k)=\theta(e_je_i,e_k)-\theta(e_j,e_ie_k),\]
for all $i,j,k \in [n].$
For the algebra $\mu_0^n$, these identities reduce to 
\[\theta(e_{i+j},e_k)=\theta(e_{i+k},e_j)\quad \mbox{and}\quad\theta(e_i,e_{j+k})=\theta(e_j,e_{i+k}), \] 
with $e_m=0$ if $m>n$.  As these are exactly the identities for the bicommutative cocycles of $\Z{BC}(\mu_0^n, \FF)$, this case also reduces to the bicommutative case.

\subsection{Left symmetric central extensions}

Recall that an algebra ${\rm A}$ is said to be left symmetric if it satisfies the identity
\[(xy)z-x(yz)=(yx)z-y(xz),\] 
for all $x,y,z\in{\rm A}$. To find all the left symmetric central extensions  of $\mu_0^n,$ we need cocycles  $\theta \colon \mu_0^n\times \mu_0^n\to \FF$ satisfying 
\[\theta(e_ie_j,e_k)-\theta(e_i,e_je_k)=\theta(e_je_i,e_k)-\theta(e_j,e_ie_k),\]
for all $i,j,k \in [n].$
Note that for the algebra $\mu_0^n$, we obtain just the relation
\[ \theta(e_i,e_{j+k})=\theta(e_j,e_{i+k}), \quad \mbox{for all $i,j,k \in [n]$},\] where $e_m=0$ if $m>n$. Thus, $\theta$ is the same as a left commutative cocycle from $\Z{LC}(\mu_0^n, \FF)$, and it follows that this case reduces to the left commutative case.



\def\cprime{$'$} \def\cprime{$'$} \def\cprime{$'$}

\end{document}